\documentclass[fleqn,reqno,11pt,a4paper,final]{amsart}

\usepackage[a4paper,left=30mm,right=30mm,top=30mm,bottom=30mm,marginpar=20mm]{geometry}
\usepackage{amsmath}
\usepackage{amssymb}
\usepackage{amsthm}
\usepackage{amscd}
\usepackage[ansinew]{inputenc}
\usepackage{cite}
\usepackage{bbm}
\usepackage{color}
\usepackage[english=american]{csquotes}
\usepackage[final]{graphicx}
\usepackage{hyperref}
\usepackage{calc}
\usepackage{mathptmx}
\usepackage{t1enc}

\graphicspath{{../Pictures/}}

\numberwithin{equation}{section}

\newtheoremstyle{thmlemcorr}{10pt}{10pt}{\itshape}{}{\bfseries}{.}{10pt}{{\thmname{#1}\thmnumber{ #2}\thmnote{ (#3)}}}
\newtheoremstyle{thmlemcorr*}{10pt}{10pt}{\itshape}{}{\bfseries}{.}\newline{{\thmname{#1}\thmnumber{ #2}\thmnote{ (#3)}}}
\newtheoremstyle{remexample}{10pt}{10pt}{}{}{\bfseries}{.}{10pt}{{\thmname{#1}\thmnumber{ #2}\thmnote{ (#3)}}}
\newtheoremstyle{ass}{10pt}{10pt}{}{}{\bfseries}{.}{10pt}{{\thmname{#1}\thmnumber{ A#2}\thmnote{ (#3)}}}

\theoremstyle{thmlemcorr}
\newtheorem{theorem}{Theorem}
\numberwithin{theorem}{section}
\newtheorem{lemma}[theorem]{Lemma}
\newtheorem{corollary}[theorem]{Corollary}

\newtheorem{conjecture}[theorem]{Conjecture}
\newtheorem{definition}[theorem]{Definition}

\theoremstyle{thmlemcorr*}
\newtheorem{theorem*}{Theorem}
\newtheorem{lemma*}[theorem]{Lemma}
\newtheorem{corollary*}[theorem]{Corollary}
\newtheorem{proposition*}[theorem]{Proposition}
\newtheorem{problem*}[theorem]{Problem}
\newtheorem{conjecture*}[theorem]{Conjecture}
\newtheorem{definition*}[theorem]{Definition}

\theoremstyle{remexample}

\theoremstyle{ass}

\def\XXint#1#2#3{{\setbox0=\hbox{$#1{#2#3}{\int}$ }
\vcenter{\hbox{$#2#3$ }}\kern-.56\wd0}}

\newcommand{\Sbb}{\mathbb{S}}
\newcommand{\T}{\mathbb{T}}

\newcommand{\Zbb}{\mathbb{Z}}

\DeclareMathOperator{\diverg}{div}
\DeclareMathOperator{\Div}{div}
\DeclareMathOperator{\curl}{curl}
\DeclareMathOperator{\dist}{dist}
\DeclareMathOperator{\rank}{rank}

\DeclareMathOperator{\supp}{supp}

\newcommand{\norm}[1]{\|#1\|}

\newcommand{\N}{\mathbb{N}}
\newcommand{\R}{\mathbb{R}}

\newcommand{\sbullet}{\begin{picture}(1,1)(-0.5,-2.5)\circle*{2}\end{picture}}
\newcommand{\frarg}{\,\sbullet\,}

\newcommand{\eps}{\epsilon}

\newcommand{\term}[1]{\textbf{#1}}

\newcommand{\vr}{\varrho}

\def\XXint#1#2#3{{\setbox0=\hbox{$#1{#2#3}{\int}$}
\vcenter{\hbox{$#2#3$}}\kern-.5\wd0}}

\renewcommand{\eps}{\varepsilon}
\renewcommand{\epsilon}{\varepsilon}
\renewcommand{\phi}{\varphi}

\begin{document}


\title[Conserved quantities and regularity in fluid dynamics]{
Conserved quantities and regularity in fluid dynamics}

\author{Emil Wiedemann}

\address{\textit{Emil Wiedemann:} Institute of Applied Analysis, Universit\"at Ulm, Helmholtzstr.\ 18, 89081 Ulm, Germany}
\email{emil.wiedemann@uni-ulm.de}

\begin{abstract}
This is a set of lecture notes for the 2019 EMS School in Applied Mathematics held in K\'acov, Czech Republic. 

Conserved or dissipated quantities, like energy or entropy, are at the heart of the study of many classes of time-dependent PDEs in connection with fluid mechanics. This is the case, for instance, for the Euler and Navier-Stokes equations, for systems of conservation laws, and for transport equations. In all these cases, a formally conserved quantity may no longer be constant in time for a weak solution at low regularity. The delicate interplay between regularity and conservation of the respective quantity relates to renormalisation in the DiPerna-Lions theory of transport and continuity equations, and to Onsager's conjecture in the realm of ideal incompressible fluids. We will review the classical commutator methods of DiPerna-Lions and Constantin-E-Titi, and then proceed to more recent results. 
\end{abstract}








\maketitle



 \tableofcontents


\section{Introduction}
Quite unsurprisingly, conserved or dissipated quantities play a fundamental role in about any evolution differential equation related to continuum mechanics. On the analytical level, these quantities are often (as in the Navier-Stokes equations) the only source of a priori estimates to yield sufficient compactness for the existence of weak solutions. In the theory of transport equations, conserved quantities allow to show uniqueness and stability through the renormalisation theory of DiPerna-Lions~\cite{dipernalions} discussed below.

On the physical level, quantities that are formally shown to be conserved can, in fact, be observed to be dissipated. This kind of anomalous dissipation occurs, e.g., in the incompressible Euler equations at low regularity due to turbulent energy transfer to high scales, as predicted by Onsager~\cite{ON} based on Kolmogorov's phenomenological theory of turbulence; hence, at low regularity, one may observe that the kinetic energy 
\begin{equation}\label{globaleuler}
\frac12\int |u(x,t)|^2dx
\end{equation}   
is actually decreasing in time, although by formal computation (as it is justified for smooth solutions) it would be constant. In fact, energy conservation formally holds in the stronger local sense
\begin{equation}\label{localeuler}
\partial_t \frac{|u|^2}{2}+\diverg\left(\left(\frac{|u|^2}{2}+p\right)u\right)=0,
\end{equation}
which implies the total conservation of energy upon integration in time and using the Divergence Theorem.

The past decades have seen increasingly sophisticated constructions giving rigorous examples of weak solutions to the Euler equations with non-conserved energy~\cite{Scheffer, shnirel1, shnirel2, euler1, euler2, eulerinvent, isett16, buckmasteretal17}, culminating in a complete proof of Onsager's Conjecture on the threshold regularity up to which anomalous dissipation can occur. This regularity is essentially at $C^{1/3}$ (the H\"older space with exponent $1/3$), and below we will at least show that $1/3$ is an upper bound. 

A more classical phenomenon in hyperbolic conservation laws is the formation of shock waves, which also leads to a decrease of the (mathematical) entropy. For instance, the inviscid Burgers equation
\begin{equation}\label{localburgers}
\partial_t u+\partial_x(u^2)=0
\end{equation}
is easily seen to satisfy the equality $\partial_t\frac{u^2}{2}+\partial_x\frac{2u^3}{3}=0$, and thus to conserve the so-called entropy $\frac12\int u^2 dx$, as long as the solution remains smooth, but it is equally easy to explicitly find an example of a weak solution that becomes discontinuous in finite time and then has strictly decreasing entropy. More precisely, given the Lipschitz initial datum
\begin{equation*}
u_0(x):=\begin{cases}
1 & \text{if $x\leq0$,}\\
1-x & \text{if $0\leq x\leq 1$,}\\
0 & \text{if $x>1$},
\end{cases}
\end{equation*}

for any time $t<1$ a solution is given by 
\begin{equation*}
u(x,t)=\begin{cases}
1 & \text{if $x\leq t$,}\\
\frac{1-x}{1-t} & \text{if $t\leq x\leq 1$,}\\
0 & \text{if $x>1$},
\end{cases}
\end{equation*}
which can be extended to $t\geq 1$ by

\begin{equation*}
u(x,t)=\begin{cases}
1 & \text{if $x<\frac{t+1}{2}$,}\\
0 & \text{if $x>\frac{t+1}{2}$},
\end{cases}
\end{equation*}
and it holds true that, after $t=1$, there are smooth test functions $\phi\geq0$ such that
\begin{equation*}
\int\int(\partial_t\phi \frac{u^2}{2}+\partial_x\phi \frac{2u^3}{3})dxdt>0,
\end{equation*}
which means the entropy conservation is violated (in accordance with the Second Law of Thermodynamics).

But even when a certain quantity is seen to be dissipative already on the formal level, it is still important to know whether it satisfies a predicted balance as an equality or not. For example, the incompressible Navier-Stokes equations formally satisfy
\begin{equation}\label{NSEbalance}
\frac 12\partial_t\int |u|^2 dx + \nu\int_0^t\int |\nabla u|^2 dxds=\frac 12\int |u_0|^2dx, 
\end{equation}  
where $\nu>0$ is the viscosity of the modelled fluid; but it can only be shown to hold as an inequality ($=$ replaced by $\leq$) for generic weak solutions of Leray-Hopf type. In contrast to the Euler equations and hyperbolic conservation laws, however, there seems to be no mathematical or physical reason to believe that~\eqref{NSEbalance} should hold with strict inequality. 

Let us condense this somewhat loose collection of observations into some general ideas that will form the focus of this survey:
\begin{itemize}
\item A great variety of partial differential equations related to continuum dynamics exhibit quantities that can be easily seen, by formal calculation invoking the chain rule of differential calculus, to be conserved in time. For smooth solutions, these calculations are easily justified rigorously.
\item These conservation laws can be formulated in a local (e.g.,~\eqref{localeuler},~\eqref{localburgers}) or in a global (e.g.,~\eqref{globaleuler},~\eqref{NSEbalance}) way.
\item On the other hand, less regular solutions may not conserve these quantities. This relates to physically observable effects of anomalous energy dissipation due to turbulence, or to increase of physical entropy due to the Second Law of Thermodynamics.
\item It is therefore worth investigating the threshold regularity below which such dissipative effects can occur. 
\end{itemize}

Our discussion here mainly focuses on inviscid models. We begin with the arguably simplest case of linear scalar conservation laws, i.e., linear transport equations, which possess infinitely many conserved quantities. The question of threshold regularity can be interpreted in several ways, depending on whether one wishes to impose regularity conditions on the coefficients alone, or on the coefficients and the solution combined. The first approach leads to the DiPerna-Lions theory of renormalisation, while the second one motivates the commutator estimates of Constantin-E-Titi~\cite{CET}. This is presented in Section~\ref{renormsec}.

The mentioned techniques give bounds from above for the sought threshold regularity (i.e., sufficient conditions for conservation), and can be viewed as restoring the chain rule in regimes of low regularity. On the other side, the construction of examples of non-conservation, and therefore of breakdown of the chain rule, can be much harder. We have seen that classical shocks provide such examples in the context of hyperbolic equations, but in incompressible models, shocks are not available. Instead, convex integration has recently become the method of choice to construct dissipative solutions of the Euler equations. In Section~\ref{chainrule}, we outline the method of convex integration in the comparatively simple setting of steady transport, and thus provide counterexamples to renormalisation.

The conservative part of Onsager's Conjecture for the incompressible Euler system forms the topic of Section~\ref{onsagersec}, including a discussion of recent results concerning bounded domains and the vanishing viscosity limit. The subsequent sections are devoted to various recent extensions of the commutator method to statistical solutions, general conservation laws, and degenerate cases such as the compressible Euler and Navier-Stokes equations with possible vacuum.

\section{Renormalisation of Transport Equations}\label{renormsec}
The first partial differential (PDE) considered in the textbook~\cite{evans} of L.~C.~Evans is also the supposedly simplest one: the \emph{transport equation}
\begin{equation}\label{transporteq}
\partial_t \rho+u\cdot\nabla \rho=0.
\end{equation} 
For further simplicity, let us consider the case of periodic boundary conditions. Then, if $\T^d:=\R^d/\Zbb^d$ is the flat torus, one usually considers the vector field $u:\T^d\times [0,T]\to\R^d$ as given and divergence-free. The scalar field $\rho:\T^d\times[0,T]\to\R$ is the unknown, which could be subject to an initial condition $\rho(\cdot,0)=\rho^0$.

The transport equation can be interpreted, for instance, as follows: The given field $u$ can be thought of as the known velocity field of an incompressible flow (hence the divergence-free condition), such as water on the surface of the ocean (in which case $d=2$). The scalar $\rho$ then gives the concentration of, say, a chemical dissolved on the ocean surface, and the chemical is transported by the given flow. As, in this very simple model, the chemical has no effect on the dynamics of the transporting flow, $\rho$ is sometimes called a \emph{passive scalar} or a \emph{passive tracer}. 

A remarkable property of the transport equation is that it can be \emph{renormalised}: Suppose an arbitrary $C^1$ function $\eta:\R\to\R$ is given, then~\eqref{transporteq} can be multiplied by $\eta'(\rho)$ to yield, by means of the chain rule,
\begin{equation}\label{renorm}
\partial_t\eta(\rho)+u\cdot\nabla \eta(\rho)=0,
\end{equation}
so that in fact arbitrary functions of a solution become a solution of the same equation. Moreover, integration in space yields $\frac{d}{dt}\int_{\T^d}\eta(\rho)dx=0$, and therefore from this renormalisation procedure we obtain infinitely many conserved quantities for~\eqref{transporteq}. 

However, the chain rule (we used it in the form $\partial_t \eta(\rho)=\eta'(\rho)\partial_t \rho$ and $\nabla\eta(\rho)=\eta'(\rho)\nabla \rho$) is only justified if $\rho$ is $C^1$, or at least Lipschitz. So the question arises:

\emph{Under what conditions on $u$ and/or $\rho$ are solutions of the transport equation renormalised?}

First of all, how does the transport even make sense if $\rho$ is not $C^1$? As usual, one considers a distributional concept of solution, using the fact that $u\cdot\nabla \rho=\Div(\rho u)$ by virtue of the divergence-free condition on $u$. Therefore, a function $\rho\in L^{\infty}(\T^d\times [0,T])$ is called a \emph{weak solution} of~\eqref{transporteq} with initial data $\rho^0\in L^\infty(\T^d)$ if, for every $\phi\in C_c^1(\T^d\times[0,T))$, we have
\begin{equation*}
\int_0^T\int_{\T^d} \partial_t\phi \rho+\rho\nabla\phi\cdot u dxdt=\int_{\T^d}\phi(x,0)\rho^0(x)dx.
\end{equation*} 
Note this definition makes sense as long as $u\in L^1_{loc}(\T^d\times[0,T))$.

Renormalisation plays an important role not only for the study of~\eqref{transporteq} or of ordinary differential equations, but also for larger systems of PDEs that contain~\eqref{transporteq} or the closely related \emph{continuity equation} 
\begin{equation}
\partial_t \rho+\Div(\rho u)=0,
\end{equation}
where $u$ is no longer assumed divergence-free. Consider two examples: The \emph{isentropic compressible Navier-Stokes system} reads
\begin{equation}\label{compNSE}
\begin{aligned}
\partial_t(\rho u)+\Div(\rho u\otimes u)+\nabla p(\rho)&=\Div\mathbb{S}(\nabla u),\\
\partial_t \rho+\Div(\rho u)&=0,
\end{aligned}
\end{equation}
where the scalar density $\rho$ is now assumed non-negative, $p$ is a given function of density, and $\Sbb$ denotes the Newtonian stress tensor. The velocity $u$ is a vector field $\T^d\times[0,T]\to\R^d$. 

Let us ignore the details of the constitutive theory for the moment and imagine $\Sbb$ as the identity, so that the right hand side can be thought of simply as $\Delta u$. It is well-known that these equations satisfy a priori bounds given by the energy inequality
\begin{equation*}
\int_{\T^d}\frac12\rho(x,t)|u(x,t)|^2+P(\rho(x,t))dx + \int_0^t\int_{\T^d}\Sbb(\nabla u):\nabla u dx ds \leq \int_{\T^d}\frac12\rho^0|u^0|^2+P(\rho^0)dx.
\end{equation*} 

Here $P$ is the so-called pressure potential given by
\begin{equation}\label{presspot}
P(r)=r\int_1^r \frac{p(s)}{s^2}ds.
\end{equation}
 Since $\Sbb(\nabla u):\nabla u\geq c|\nabla u|^2$, we obtain a bound for $u$ in $L^2(0,T;H^1(\T^d))$, whereas for $\rho$ we obtain only integrability, but no regularity properties. In fact, this is to be expected since~\eqref{compNSE} is parabolic in the momentum equation but only hyperbolic in the mass equation.

The question whether the density can be renormalised plays an important role, e.g., in the Lions-Feireisl theory of weak solutions for~\eqref{compNSE}, see~\cite{lions, feireisl}. As we just saw, we obtain a priori information on the regularity of the transporting velocity, but not on the transported scalar. In this context, one thus asks: Under what (Sobolev) regularity assumptions on the transporting velocity field $u$ is every bounded weak solution $\rho$ of~\eqref{transporteq} renormalised (i.e.\ every smooth function of $\rho$ is again a weak solution of~\eqref{transporteq})? This is the subject of the famous DiPerna-Lions theory that we will outline shortly.

As another example, consider an \emph{active scalar equation} of the form
\begin{equation}
\begin{aligned}
\partial_t \rho+u\cdot \nabla\rho &=0,\\
\Div u&=0,\\
u=T[\rho],
\end{aligned}
\end{equation}
where $T$ is a Fourier multiplier operator of order zero. For instance, if (for $d=2$) the symbol of $T$ is given by $i\frac{\xi^\perp}{|\xi|}$, we obtain the well-known surface quasi-geostrophic (SQG) equation.

The only thing of interest at the moment, anyway, is the fact that $T$ is bounded from $L^p$ to $L^p$ ($1<p<\infty$) and therefore, the active scalar $\rho$ and the velocity $u$ have the same (Lebesgue, Besov, H\"older etc.) regularity. In contrast to the Navier-Stokes equations, therefore, we can not (or do not have to) distinguish between the regularities of the scalar and the vector field.

A similar example is given by the \emph{incompressible Euler equations}
\begin{equation}\label{inceuler}
\begin{aligned}
\partial_t u+(u\cdot \nabla)u+\nabla p&=0\\
\Div u&=0,
\end{aligned}
\end{equation} 
where in some sense the velocity is transported by itself (this is the effect of \emph{advection}). Again, there is no way to distinguish between the regularities of the transporting and the transported quantity. We will get back to the Euler equations later.

\subsection{DiPerna-Lions commutators}
Again let us simply consider the transport equation~\eqref{transporteq} and assume that $\rho, u$ form a weak solution (where $u$ is still divergence-free and $\rho$ is bounded for simplicity). In cases where some regularity is known for $u$ (loosely speaking, at least one full distributional space derivative), but none for $\rho$, the theory of DiPerna-Lions is useful.

Let $\chi:\T^d\to\R$ be a standard mollifier, i.e.\ a smooth non-negative radially symmetric function with compact support in $B_1(0)$ and $\int_{B_1(0)}\chi dx=1$. Set $\chi_\epsilon(x)=\epsilon^{-d}\chi\left(\frac x\epsilon\right)$, which is supported on $B_\epsilon(0)$ and still has unit integral. For a function $f$, We write $f_\epsilon:=f*\chi_\epsilon$.

Mollifying~\eqref{transporteq}, we obtain (ignoring issues of time differentiability)
\begin{equation}\label{molltransport}
0=\partial_t\rho_\epsilon+\Div(\rho u)_\epsilon=\partial_t\rho_\epsilon+\Div(\rho_\epsilon u)+R_\epsilon,
\end{equation}
where 
\begin{equation}
R_\epsilon=\Div(\rho u)_\epsilon-\Div(\rho_\epsilon u)
\end{equation}
is the \emph{commutator}. Let now $\eta\in C^1$ and multiply~\eqref{molltransport} by $\eta'(\rho_\epsilon)$. Since $\rho_\epsilon$ is smooth, we can apply the chain rule to obtain
\begin{equation}
\partial_t\eta(\rho_\epsilon)+ \Div(\eta(\rho_\epsilon) u)=-\eta'(\rho_\epsilon)R_\epsilon
\end{equation}  
(recall that $u$ is divergence-free, so that $\Div(\rho_\epsilon u)=u\cdot\nabla \rho_\epsilon$). First, it is clear that the terms on the left hand side converge, in the sense of distributions as $\epsilon\to0$, to $\partial_t\eta(\rho)+ \Div(\eta(\rho) u)$. Hence, as $\eta'(\rho_\epsilon)$ is bounded uniformly in $\epsilon$, it suffices to show that $R_\epsilon \to 0$ in $L^1(\T^d\times[0,T])$.

To this end, we compute
\begin{equation*}
\begin{aligned}
R_\epsilon(x,t)&=\Div(\rho u)*\chi_\epsilon(x,t)-\Div((\rho*\chi_\epsilon) u(x,t))\\
&=-\rho u*\nabla\chi_\epsilon(x,t)+\rho*\nabla\chi_\epsilon\cdot u(x,t)\\
&=\int_{B_\epsilon(x)} \rho(y,t)(u(x,t)-u(y,t))\cdot \nabla\chi_\epsilon(x-y)dy\\
&=\epsilon^{-d-1}\int_{B_\epsilon(x)} \rho(y,t)(u(x,t)-u(y,t))\cdot \nabla\chi\left(\frac{x-y}{\epsilon}\right)dy\\
&=-\int_{B_1(0)} \rho(x+\epsilon z,t)\frac{u(x+\epsilon z,t)-u(x,t)}{\epsilon}\cdot \nabla\chi\left(z\right)dz,
\end{aligned}
\end{equation*}
where in the end we used the transformation $z=\frac{y-x}{\epsilon}$.

Suppose that $u\in L^1(0,T;W^{1,1}(\T^d))$, then, by standard difference quotient lemmas, $\frac{u(x+\epsilon z,t)-u(x,t)}{\epsilon}$ converges, as $\epsilon\to0$, to the directional derivative $\partial_z u(x)$ in $L^1(\T^d\times (0,T))$ for fixed $z$; moreover it is bounded in $L^1$ uniformly in $z$ and $\epsilon$. Since $\rho(x+\epsilon z)$ is in $L^\infty$, uniformly in $\epsilon$ and $z$, and converges in $L^1$ to $\rho(x)$ for fixed $z$, we obtain the strong $L^1$ convergence
\begin{equation}
\begin{aligned}
R_\epsilon(x,t)&\to-\rho(x,t)\int_{B_1(0)} \partial_zu(x,t)\cdot \nabla\chi\left(z\right)dz\\
&=-\rho(x,t)\partial_j u_i(x,t)\int_{B_1(0)} z_j\partial_i\chi\left(z\right)dz\\
&=\rho(x,t)\delta_{ij}\partial_ju_i(x,t)=\rho(x,t)\Div u(x,t)=0,
\end{aligned}
\end{equation} 
as desired.

As the convergence argument may not be obvious, let us give it in more detail: Write $D_\epsilon(x,z,t):=\frac{u(x+\epsilon z,t)-u(x,t)}{\epsilon}$ and $\rho_\epsilon(x,z,t):=\rho(x+\epsilon z,t)$, then
\begin{equation*}
\begin{aligned}
&\int_{\T^d\times(0,T)}\left|R_\epsilon(x,t)+\int_{B_1(0)}\partial_zu(x,t)\cdot \nabla\chi\left(z\right)dz\right|dxdt\\
&\leq \|\nabla\chi\|_\infty\left(\int_{\T^d\times(0,T)}\int_{B_1(0)}|\rho_\epsilon(x,z,t)D_\epsilon(x,z,t)-\rho_\epsilon(x,z,t)\partial_zu(x,t)|dzdxdt\right.\\
&\hspace{2cm}+\left.\int_{\T^d\times(0,T)}\int_{B_1(0)}|\rho_\epsilon(x,z,t)\partial_z u(x,t)-\rho(x,t)\partial_z u(x,t)|dzdxdt\right)\\
&\leq \|\nabla\chi\|_\infty\sup_{\epsilon, z}\|\rho_\epsilon\|_\infty\int_{B_1(0)}\int_{\T^d\times(0,T)}|D_\epsilon-\partial_z u|dxdz\\
 &\hspace{2cm}+  \|\nabla\chi\|_\infty\sup_{z}\|\partial_z u\|_{L^1}\int_{B_1(0)}\int_{\T^d\times(0,T)}|\rho_\epsilon-\rho|dx dz.
\end{aligned}
\end{equation*}
As mentioned above, for each $z$, $\int_{\T^d\times(0,T)}|D_\epsilon-\partial_z u|dx$ converges to zero as $\epsilon\to0$, and the same is true for $\int_{\T^d\times(0,T)}|\rho_\epsilon-\rho|dx$. Since these expressions are dominated, respectively, by 
\begin{equation*}
2\sup_z\int_{\T^d\times(0,T)}|\partial_z u(x)|dx,\quad 2\int_{\T^d\times(0,T)}|\rho(x)|dx, 
\end{equation*}
which are constant and hence integrable in $z$, we conclude by the Dominated Convergence Theorem.

We have thus proved:
\begin{theorem}[DiPerna-Lions]
If $u\in L^1(0,T;W^{1,1}(\T^d))$, then every bounded weak solution of~\eqref{transporteq} is renormalised in the sense of~\eqref{renorm}.
\end{theorem}
Of course this result holds in much greater generality (the assumptions $\rho\in L^\infty$ and $\Div u=0$ can be substantially relaxed). In fact it follows easily from renormalisation that the Cauchy problem for~\eqref{transporteq} admits a unique weak solution.

\subsection{Constantin-E-Titi commutators}
Let us present a commutator argument that looks very similar to the previous one, but leads to very different conclusions. 
Once again we mollify equation~\eqref{transporteq} in space,
\begin{equation}
0=\partial_t\rho_\epsilon+\Div(\rho u)_\epsilon=\partial_t\rho_\epsilon+\Div(\rho_\epsilon u_\epsilon)+S_\epsilon,
\end{equation} 
with 
\begin{equation}
S_\epsilon=\Div(\rho u)_\epsilon-\Div(\rho_\epsilon u_\epsilon).
\end{equation}
Note the only difference compared to~\eqref{molltransport} is that we chose to mollify also $u$. Multiplying again by $\eta'(\rho_\epsilon)$, we obtain (noting that $u_\epsilon$ is still divergence-free)
\begin{equation}
\partial_t\eta(\rho_\epsilon)+\Div(\eta(\rho_\epsilon)u_\epsilon)=-\eta'(\rho_\epsilon)S_\epsilon,
\end{equation}
so we obtain renormalisation provided we can show the the right hand side converges to zero, in the sense of distributions, as $\epsilon\to0$.  

To this end, let $\phi\in C_c^1(\T^d\times(0,T))$, so that integration by parts yields
\begin{equation}
\begin{aligned}
-\int_0^T\int_{\T^d}\phi\eta'(\rho_\epsilon)S_\epsilon dxdt &= \int_0^T\int_{\T^d}\eta'(\rho_\epsilon)\nabla\phi\cdot ((\rho u)_\epsilon-\rho_\epsilon u_\epsilon) dxdt\\
& + \int_0^T\int_{\T^d}\phi\eta''(\rho_\epsilon)\nabla\rho_\epsilon\cdot ((\rho u)_\epsilon-\rho_\epsilon u_\epsilon) dxdt.
\end{aligned}
\end{equation}
We only treat the second integral, as the first one is easier. Assume to this end that $\eta''$ is bounded, and consider the pointwise identity
\begin{equation}
(\rho u)_\epsilon-\rho_\epsilon u_\epsilon = -(\rho_\epsilon-\rho)(u_\epsilon-u)+\int_{B_\epsilon(0)} \chi_\epsilon(y)(\rho(\cdot-y)-\rho)(u(\cdot-y)-u)dy.
\end{equation}
Thus, one part of the desired estimate is obtained by
\begin{equation}
\int_0^T\int_{\T^d}|\nabla\rho_\epsilon| |\rho_\epsilon-\rho||u_\epsilon-u| dxdt\leq \|\nabla\rho_\epsilon\|_{L^p}\|\rho-\rho_\epsilon\|_{L^p}\|u-u_\epsilon\|_{L^q}
\end{equation}
for some exponents satisfying $\frac2p+\frac1q\leq1$. The other part can be estimated similarly, so we ignore it.

The question arises under what conditions these norms converge to zero as $\epsilon\to0$. To this end, suppose that
\begin{equation}\label{BesovVMO}
\lim_{\epsilon\to0} \int_0^T\int_{\T^d}\int_{B_\epsilon(0)}\frac{|\rho(x)-\rho(x-y)|^p}{\epsilon^{d+\alpha p}}dydxdt=0.
\end{equation}
We aim to show that, under this assumption, $\lim_{\epsilon\to0}\epsilon^{-\alpha}\|\rho-\rho_\epsilon\|_{L^p}=0$. Indeed, using Jensen's inequality and the definition of $\chi_\epsilon$,
\begin{equation}
\begin{aligned}
\epsilon^{-\alpha p}\|\rho-\rho_\epsilon\|_{L^p}^p&=\int_0^T\int_{\T^d}\left|\int_{B_\epsilon(0)}\frac{\rho(x)-\rho(x-y)}{\epsilon^\alpha}\chi_\epsilon(y)dy\right|^pdxdt\\
&\leq \int_0^T\int_{\T^d}\int_{B_\epsilon(0)}\frac{|\rho(x)-\rho(x-y)|^p}{\epsilon^{d+\alpha p}}\chi\left(\frac y\epsilon\right)dydxdt,
\end{aligned}
\end{equation}
which converges to zero by virtue of assumption~\eqref{BesovVMO}. Likewise, we have $\lim_{\epsilon\to0}\epsilon^{-\beta}\|u-u_\epsilon\|_{L^q}=0$ provided
\begin{equation}\label{BesovVMO2}
\lim_{\epsilon\to0} \int_0^T\int_{\T^d}\int_{B_\epsilon(0)}\frac{|u(x)-u(x-y)|^q}{\epsilon^{d+\beta q}}dydxdt=0.
\end{equation}
Finally, it is not difficult to get the estimate $\lim_{\epsilon\to0}\epsilon^{1-\alpha}\|\nabla\rho_\epsilon\|_{L^p}=0$.
Assume now that $2\alpha+\beta\geq 1$, then 
\begin{equation}
\begin{aligned}
\|\nabla\rho_\epsilon\|_{L^p}&\|\rho-\rho_\epsilon\|_{L^p}\|u-u_\epsilon\|_{L^q}\leq\\
&\epsilon^{1-\alpha}\|\nabla\rho_\epsilon\|_{L^p}\epsilon^{-\alpha}\|\rho-\rho_\epsilon\|_{L^p}\epsilon^{-\beta}\|u-u_\epsilon\|_{L^q}\to0,
\end{aligned}
\end{equation}
and so we arrive at the following result:
\begin{theorem}\label{transportthm}
Let $\frac 2p+\frac 1q\leq 1$ and $2\alpha+\beta\geq 1$, and $\rho\in L^p(\T^d\times(0,T))$ and $u\in L^q(\T^d\times(0,T))$ be a weak solution of~\eqref{transporteq} satisfying~\eqref{BesovVMO} and~\eqref{BesovVMO2}. If $\eta\in W^{2,\infty}(\R)$, then the renormalised equation holds in the sense of distributions:
\begin{equation}
\partial_t\eta(\rho)+u\cdot \nabla \eta(\rho)=0.
\end{equation}
\end{theorem}
Several remarks are in order. This commutator method relies on the ideas in~\cite{CET} in the context of Onsager's conjecture (which we will get back to later). In the case of Euler, the transporting field and the transported quantity are identical, which forces $p=q\geq 3$ and $\alpha=\beta\geq \frac13$. Mixed regularities for $\rho$ and $u$ were introduced in~\cite{LSh, FGSW}. Theorem~\ref{transportthm} was proved in~\cite{AkWi}, but in a slightly more restrictive functional framework. The paper~\cite{AkWi} also relaxes the condition $\eta\in W^{2,\infty}(\R)$, and gives an application to active scalar equations. Conditions~\eqref{BesovVMO} and~\eqref{BesovVMO2} were first formulated in~\cite{FjWi} and thus generalised previous works like~\cite{CET, duchonrobert, CCFS08}. They are implied, for instance, by $\rho\in L^p((0,T);C^{\alpha'}(\T^d))$ or $\rho\in L^p((0,T);B_{p,\infty}^{\alpha'}(\T^d))$ for any $\alpha'>\alpha$, where $C^{\alpha'}$ and $B_{p,\infty}^{\alpha'}$ denote the respective H\"older and Besov spaces, and similarly for $u$.

\section{Counterexamples to Renormalisation via Convex Integration}\label{chainrule}
It may seem strange that renormalisation could fail for a simple linear equation like~\eqref{transporteq}. But in fact, there is a great deal of counterexamples known, starting from the original paper of DiPerna-Lions and extended in various ways in~\cite{depauw, shvydkoy, AlBiCr, IsVi, CGSW2, CGSW1, MSz, MSz2, MSa}. A powerful instrument to produce rough and often pathological solutions to certain PDEs is known as \emph{convex integration}. We present here the main ideas from~\cite{CGSW1} (full details can be found in that paper). The result discussed here is by no means optimal in terms of regularity of the velocity field, but it does show that renormalisation can fail in basically any conceivable way, and it gives an idea of the general technique of convex integration in a comparatively simple setting.

We consider the stationary problem in 3D, and we aim to show the following:
\begin{theorem}
Let $f\in\mathcal{D}'(\T^3;\R)$ be a distribution such that there exists a bounded continuous solution of the equation $\Div w=f$. Then, there exist a bounded vectorfield $u\in L^\infty(\T^3;\R^3)$ and a bounded scalar field $\rho\in L^\infty(\T^3)$ such that
\begin{equation}\label{original}
\begin{aligned}
\Div(\rho u)&=0\\
\Div(u)&=0\\
\Div(\rho^2 u)&=f.
\end{aligned}
\end{equation} 
\end{theorem} 
Note once more that if the chain rule were valid, we would have $\Div(\rho^2 u)=0$. Therefore, $f$ is called the \emph{renormalisation defect}.

We begin by vastly relaxing the problem to the study of triplets $(m,u,w)$ of vectorfields that satisfy
\begin{equation}\label{linear}
\begin{aligned}
\Div(m)&=0\\
\Div(u)&=0\\
\Div(w)&=f.
\end{aligned}
\end{equation}
Setting 
\begin{equation}\label{nonlinear}
\begin{aligned}
K_{C}&:=\left\{(m,u,w)\in\R^{3\times3}: \frac{1}{C}\leq|u|\leq C \right.\\
&\left.\text{ and there is $\frac{1}{C}\leq\rho\leq C$ such that } m=\rho u, w=\rho^2u\right\},
\end{aligned}
\end{equation} 
we realise that a solution of~\eqref{linear} will be a solution of the original problem~\eqref{original} if $(m,u,w)(x)\in K_C$ for almost every $x\in\T^3$, for some $C>1$.

Our strategy is now roughly as follows: 
\begin{enumerate}
\item Start with the ``subsolution'' $U^0:=(0,0,w)$, where $w$ has the property $\Div w=f$.
\item For almost every $x\in\T^3$, represent $(0,0,w)$ as a convex combination of the form $\sum_{i=1}^N\lambda_i(m_i,u_i,w_i)$ in a way that is compatible with~\eqref{linear}, and such that the $(m_i,u_i,w_i)$ are in $K_C$.
\item Find a new triplet of functions $(m^1,u^1,w^1)$ that still satisfies~\eqref{linear}, and that approximately takes the value $(m_i,u_i,w_i)$ on a volume fraction $\lambda_i$ near the point $x$.
\item Iterate the process to obtain a sequence $(m^n,u^n,w^n)_{n\in\N}$ of functions satisfying~\eqref{linear}, and show that the sequence converges to a solution of~\eqref{original}.  
\end{enumerate}

Step 1 can be done by assumption. For Step 2, it surely should be clarified what is meant by ``compatibility'' of the convex combination with constraint~\eqref{linear}. Clearly, without a differential constraint like~\eqref{linear}, it is easy to find a continuous map $U^1:=(m^1,u^1,w^1):\T^3\to \R^{3\times 3}$ such that, close to any point $x\in\T^3$, $U^1$ takes values close to $U_i:=(m_i,u_i,w_i)$ on a set of volume fraction $\lambda_i$; suppose for simplicity that the $\lambda_i$ and $U_i$ are constant, then such a map could simply be given as (a suitable mollification of) the piecewise constant map
\begin{equation}
U(x)=U_n \quad\text{if $x_1\in\left(\sum_{i=1}^{n-1}\lambda_i, \sum_{i=1}^{n}\lambda_i\right)$}.
\end{equation}  
Now of course $\lambda_i$ and $U_i$ do depend on $x$, but as this is the case continuously, we can cover the domain by small cubes, consider $\lambda_i$ and $U_i$ constant on each such cube, and glue the various pieces together.

However, we want $U-(0,0,w)$ to be divergence-free, so that~\eqref{linear} remains valid. Suppose, at a given point $x\in\T^3$, $U^0(x)$ can be represented as a convex combination of only two matrices, $U^0(x)=\lambda U_1+(1-\lambda)U_2$. Does there exist a divergence-free matrix field $U$ that oscillates between $U_1$ and $U_2$? 

To this end, we define the \emph{wave cone} for the divergence-free condition as
\begin{equation}\label{lambdaconvex}
\Lambda:=\{\bar{U}\in\R^{3\times 3}: \text{there exists $\xi\neq 0$ such that $\Div [h(x\cdot \xi)\bar{U}]=0$ for any $h:\R\to\R$}\}.
\end{equation}
Since $\bar{U}\in\Lambda$ is equivalent to $h'(x\cdot\xi)\bar{U}\xi=0$ for all $h$ and a non-trivial $\xi$, we get that $\bar{U}\in\Lambda$ if and only if $\det\bar{U}=0$ or, in other words, the rank of $\bar{U}$ is at most 2. Therefore, there exists an oscillation between $U_1$ and $U_2$ if and only if $\rank(U_2-U_1)\leq 2$. 

Inductively, one can then find a condition for divergence-free oscillations between more than two matrices:

\begin{definition}\label{rk2laminate}
Suppose $\lambda_i>0$ for $i=1,\ldots,n$, $\sum_{i=1}^{n}\lambda_i=1$, and $U_i\in\R^{3\times3}$ for $i=1,\ldots,n$. The family of pairs $(\lambda_i,U_i)_{i=1}^n$ satisfies the (inductively defined) \term{$H_n$-condition} if
\begin{itemize}
\item[i)] $\rank(U_2-U_1)\leq2$ in the case $n=2$;
\item[ii)] after a relabeling of indices, if necessary, we have $\rank(U_2-U_1)\leq2$ and the family $(\tau_i,V_i)_{i=1}^{n-1}$ satisfies the $H_{n-1}$-condition, where
\begin{equation*}
\begin{aligned}
\tau_1=\lambda_1+\lambda_2, \hspace{0.2cm}\tau_i=\lambda_{i+1}\hspace{0.2cm}\text{for $i=2,\ldots,n-1$}
\end{aligned}
\end{equation*}
and
\begin{equation*}
\begin{aligned}
V_1=\frac{\lambda_1}{\tau_1}U_1+\frac{\lambda_2}{\tau_1}U_2, \hspace{0.2cm}V_i=U_{i+1}\hspace{0.2cm}\text{for $i=2,\ldots,n-1$}
\end{aligned}
\end{equation*}
in the case $n>2$.
\end{itemize}
Moreover we adopt the convention that every pair of the form $(1,U)$ satisfies the $H_1$-condition.

\end{definition}  

With this definition, we can thus give a precise meaning to Step (2) above: For almost every $x\in\T^3$, represent $U^0:=(0,0,w(x))$ as a convex combination of the form $\sum_{i=1}^N\lambda_i(x)U_i(x)$, where $(\lambda_i(x),U_i(x))_{i=1}^N$ satisfy the $H_N$-condition, and $U_i(x)\in K_C$.

That this can actually be done is the content of the following lemma. Note that, since $w$ is assumed bounded and continuous, we may always assume $|w|\geq 1$, as adding a constant will not affect the property $\Div w=f$.
\begin{lemma}\label{geom1}
Let $U=(0,0,w)\in\R^{3\times3}$ such that $|w|\geq1$. Then there exist $(\lambda_i,{U_i})_{i=1}^3$ satisfying the $H_3$-condition, such that
\begin{equation*}
U=\sum_{i=1}^3\lambda_iU_i,
\end{equation*}
and a number $C>1$ such that
\begin{equation*}
U_i\in K_{C} \quad\text{for $i=1,2,3$.}
\end{equation*}
\end{lemma}
\begin{proof}
We split $(0,0,w)$ into
\begin{equation*}
(0,0,w)=\frac{1}{2}\left(-w,-w,w\right)+\frac{1}{2}\left(w,w,w\right).
\end{equation*}  
If we call the matrices corresponding to the two triplets on the right hand side $U_-$ and $U_+$, respectively, we first observe that $U_-$ and $U_+$ are rank-2 connected since $(U_--U_+)e_3=0$. Secondly, $U_+\in K_{C}$ for any $C$ such that
\begin{equation*}
C\geq|w|.
\end{equation*}

Next, let us further decompose $U_-$. We make the ansatz 
\begin{equation}\label{decomp}
\left(-w,-w,w\right)=\frac{1}{2}(\rho_1v_1,v_1,\rho_1^2v_1)+\frac{1}{2}(\rho_2v_2,v_2,\rho_2^2v_2)
\end{equation}
with
\begin{equation}\label{decomp2}
v_1=w,\hspace{0.2cm}v_2=-3w.
\end{equation}
Then clearly~\eqref{decomp} is a rank-2 decomposition (in fact even rank-1), and~\eqref{decomp} and~\eqref{decomp2} result in the conditions
\begin{equation}\label{rhosystem}
\begin{aligned}
-\rho_1+3\rho_2&=2\\
-\rho_1^2+3\rho_2^2&=-2.
\end{aligned}
\end{equation}
A direct calculation shows that these requirements can be satisfied by two numbers $\rho_1,\rho_2>1$, and the proof is finished.
\end{proof}  

We proceed to Step (3). This consists in finding a continuous divergence-free map $U^1=(m^1,u^1,w^1):\T^3\to\R^{3\times 3}$ that approximately ``recovers'' the values $U_i(x)$ from the previous step. Given a family $(\lambda_i,U_i)_{i=1}^N$ satisfying the $H_N$-condition, we call the probability measure 
\begin{equation}
\nu:=\sum_{i=1}^N\lambda_i\delta_{U_i}\in\mathcal{P}(\R^{3\times 3})
\end{equation} 
a \emph{rank-2 laminate}, and if the $\lambda_i$ and $U_i$ depend on $x$, then we obtain a probability measure $\nu_x$ that depends on $x\in\T^3$. (Such measures are often called \emph{parametrised probability measures} or \emph{Young measures}).

If $K\in\R^{3\times 3}$ is compact, we denote by $K^{2lc}\supset K$ the \emph{rank-2 lamination convex hull} of $K$, that is, the set of barycentres of rank-2 laminates supported on $K$. In other words, if $U\in K^{2lc}$, then $U$ is the convex combination of some matrices in $K$ that satisfy an $H_N$-condition.

The recovery lemma can then be stated as follows:
\begin{lemma}\label{approx}
Let $K\subset\R^{3\times 3}$ be compact and $(\nu_x)_{x\in\Omega}$ be a family of probability measures such that
\begin{itemize}
\item[a)] the measure $\nu_x$ is a rank-2 laminate for almost every $x\in\Omega$,
\item[b)] $\supp\nu_x\subset K$ for almost every $x$.
\end{itemize}
Assume further that $\psi\in C(\R^{3\times3};\R)$ is a non-negative function that vanishes on $K$. Then the barycentre $\bar{\nu}_x=\int_K V d\nu_x(V)$ is well-defined for almost every $x\in\Omega$, and for every $\epsilon>0$ there exists a matrix-valued function $U$ such that
\begin{itemize}
\item[i)] $\Div U=\Div\bar{\nu}$ \hspace{0.2cm} in the sense of distributions,
\item[ii)] \begin{equation*}
\int_{\Omega}\psi(U(x))dx<\epsilon,
\end{equation*}
\item[iii)] \begin{equation*}
\norm{\dist(U(x),K^{2lc})}_{L^{\infty}(\Omega)}<\epsilon,
\end{equation*}
\item[iv)]
\begin{equation}\label{expectationclose}
\int_{\Omega}\left|U(x)-\bar{\nu}_x\right|dx<\int_{\Omega}\int_{\R^{3\times3}}\left|V-\bar{\nu}_x\right|d\nu_x(V)dx+\epsilon.
\end{equation}
\end{itemize}
Moreover, if $\bar{\nu}\in C(\bar{\Omega})$, then $U$ can be chosen to satisfy $U\in C(\bar{\Omega})$.
\end{lemma} 

The proof of this lemma is largely standard. By translation and localisation arguments, one reduces to the situation where $\nu$ is independent of $x$, and uses nested oscillations along rank-2 lines, as provided by~\eqref{lambdaconvex}. Note that all of this requires a localisation technique, i.e.\ a way to replace a divergence-free matrix field with another one of compact support, such that it is close to the original one on its support. It would be tempting to multiply a given field $U$ by a compactly supported cutoff function $\phi:\T^3\to\R$, but then $\phi U$ might not be divergence-free; instead, since $\Div U=0$, there exists a potential $\Psi:\T^3\to\R^{3\times 3}$ such that $\curl\Psi=U$, and one localises considering $\tilde{U}:=\curl(\phi\Psi)$. We omit details.

In this way, we get a new matrix field $U^1=(m^1,u^1,w^1)$ that is much closer to $K_C$ that the original $U^0=(0,0,w)$. However, as Lemma~\ref{approx} still unavoidably contains errors of order $\epsilon$, we have to repeat the procedure (Step (4)). Since $U^1$ is no longer of the special form $(0,0,w)$, we need a more general version of Lemma~\ref{geom1}:

\begin{lemma}\label{geom2}
Let $\epsilon>0$ and $\tilde{C}>1$. There exists a strictly increasing continuous function $h:[0,\infty)\to[0,\infty)$, depending only on $\tilde{C}$, with $h(0)=0$, and a number $\delta>0$, depending only on $\tilde{C}$ and $\epsilon$, such that for every $1<C<\tilde{C}-\epsilon$ and every $U\in\R^{3\times3}$ such that $\dist(U,K^{2lc}_{C})<\delta$, there exists a rank-2 laminate $\nu=\sum_{i=1}^n\lambda_i\delta_{U_i}$ such that
\begin{equation}\label{expectation}
U=\sum_{i=1}^n\lambda_iU_i,
\end{equation}
\begin{equation}\label{L1estimate}
\sum_{i=1}^n\lambda_i|U_i-U|\leq h\left(\dist(U,K_{C})\right),
\end{equation}
and
\begin{equation*}
\supp\nu\subset K_{C+\epsilon}.
\end{equation*}
\end{lemma}
The proof of this lemma is actually at the core of the whole construction, but we shall not discuss it here since the proof of Lemma~\ref{geom1} hopefully already gives a taste for the geometric arguments involved.

To finish up the proof, we need to define a sequence $(U^n)_{n\in\N}$ of divergence-free matrix fields whose average distance to $K_C$ converges to zero, for some $C>1$.

To this end, let $C_0>1$ be as required by Lemma~\ref{geom1} applied to $U^0(x)$ for \emph{all} $x\in\bar{\Omega}$ (this is possible since $U^0$ is bounded). Next, pick a sequence $(C_n)_{n\geq0}$ that is strictly increasing such that $C_n\nearrow C_0+1=:C$ as $n\to\infty$. We also set $\epsilon_n:=C_{n+1}-C_n$. Then, $(\epsilon_n)$ is a sequence of positive numbers converging to zero.

By Lemma~\ref{geom1} there exists for almost every $x\in\T^3$ a rank-2 laminate $\nu^0_x$ of finite order whose expectation is $U^0(x)$ and whose support is contained in $K_{C_0}$. This completes the definition of $U^0$ and $\nu^0$.

Suppose now that $U^n$ and $\nu^{n}$ have already been constructed for some $n\geq0$ in such a way that $\supp\nu^n\subset K_{C_n}$ and~\eqref{linear},~\eqref{expectation},~\eqref{L1estimate} are satisfied, that is:
\begin{equation*}
\Div(U^n)=(0,0,f)^T,
\end{equation*}
\begin{equation}\label{nexpectation}
U^{n}(x)=\bar{\nu}_x^n,
\end{equation}
\begin{equation*}
\int_{\R^{3\times3}}|V-U^n(x)|d\nu_x^n(V)\leq h\left(\dist(U^n,K_{C_{n-1}})\right).
\end{equation*}
The last estimate is claimed only for $n\geq1$. By Lemma~\ref{geom2}, where we set $\epsilon=\epsilon_{n+1}$ and $\tilde{C}=C+1$, there exists $\delta_{n+1}=\delta(\epsilon_{n+1})$ such that whenever
\begin{equation*}
\dist(U,K^{2lc}_{C_n})<\delta_{n+1},
\end{equation*}
then there exists a rank-2 laminate whose expectation is $U$ and whose support is contained in
\begin{equation}\label{n+1}
K_{C_n+\epsilon_{n+1}}\subset K_{C_{n+1}}.
\end{equation}
Therefore we apply Lemma~\ref{approx} to $(\nu^n_x)$ with $K_{C_n}$, $\epsilon=\delta_{n+1}$, and 
\begin{equation*}
\psi=h\left(\dist(\frarg,K_{C_n})\right).
\end{equation*}
This yields a matrix field $U_{n+1}$ satisfying 
\begin{equation*}
\Div(U_{n+1})=\Div\left(\bar{\nu}^n_x\right)=\Div(U_n)=(0,0,f)^T,
\end{equation*}
\begin{equation}\label{Kapprox}
\int_{\Omega}h\left(\dist(U_{n+1}(x),K_{C_n})\right)dx<\delta_{n+1},
\end{equation}
and
\begin{equation}\label{distn+1}
\norm{\dist(U_{n+1}(x),K_{C_n}^{2lc})}_{L^{\infty}(\Omega)}<\delta_{n+1}.
\end{equation}
Therefore, by~\eqref{n+1}, we can indeed find, for every $x$, a rank-2 laminate $\nu^{n+1}_x$ with support in $K_{C_{n+1}}$ satisfying~\eqref{expectation} and~\eqref{L1estimate}. This completes the construction of the sequence $(U^n)$.

It is then not hard to show that $U^n\to U$ in $L^1(\T^3;\R^{3\times 3})$, and that consequently $U(x)\in K_C$ for almost every $x$. Note that, in the construction of our sequence $(U^n)$, the admissible error $\epsilon$ for Lemma~\ref{approx} is chosen as $\delta_{n+1}$, which typically is much smaller than the previous $\epsilon$. This choice forces $U^{n+1}-U^n$ to oscillate at a much higher frequency that $U^n$ itself. This separation of frequencies, as one might call it, is typical of any convex integration type argument.

\section{Onsager's Conjecture}\label{onsagersec}
\subsection{Onsager's conjecture on $\T^3$}
Consider again the incompressible Euler equations,
\begin{equation}
\begin{aligned}
\partial_t u+(u\cdot\nabla)u+\nabla p&=0\\
\Div u&=0.
\end{aligned}
\end{equation}   
Here, $u:\T^3\times[0,T]\to\R^3$ denotes the velocity and $p:\T^3\times[0,T]\to\R$ the pressure of an ideal (i.e.\ inviscid) incompressible fluid, like approximately water. Multiplying this system by its velocity and integrating in space gives 
\begin{equation}
\frac{d}{dt}\int_{\T^3}|u|^2dx+\int_{\T^3}(u\cdot\nabla)u\cdot udx+\int_{\T^3}u\cdot\nabla p dx=0.
\end{equation}
The last integral vanishes since $u$ is divergence-free, and for the middle one we compute, by integration by parts,
\begin{equation}
\int_{\T^3}(u\cdot\nabla)u\cdot udx=\int_{\T^3}u_j\partial_ju_iu_i dx=-\int_{\T^3}\partial_ju_ju_iu_i dx-\int_{\T^3}u_ju_i\partial_ju_i dx,
\end{equation}
and since the first integral on the right hand side again vanishes due to the divergence-free condition, it follows that $\int_{\T^3}(u\cdot\nabla)u\cdot udx=0$, so we are left with the conservation of energy,
\begin{equation}\label{globalen}
\frac{d}{dt}\int_{\T^3}|u|^2dx=0.
\end{equation}
In fact, a more careful computation (that omits integration in space) yields the \emph{local energy equality}
\begin{equation}\label{localen}
\partial_t\frac{|u|^2}{2}+\left(\left(\frac{|u|^2}{2}+p\right)u\right)=0,
\end{equation}
which is the precise analogue (for $\eta(u)=\frac12|u|^2$) of the renormalised equation~\eqref{renorm} in the context of transport equations.

Note once more that the chain rule (or, on a related note, integration by parts) requires $u$ to have at least one full derivative. On the other hand, there are examples known of weak solutions that grossly violate~\eqref{localen} and even~\eqref{globalen}, the first one being due to Scheffer~\cite{Scheffer}, whose solutions are no better than $L^2_{loc}$. The question thus arises whether there exists a threshold regularity that distinguishes dissipative from conservative solutions. In 1949, L.~Onsager made the following conjecture~\cite{ON}:
\begin{conjecture}
\begin{itemize}
\item[a)] If $u$ is a weak solution of the incompressible Euler equations with $u\in C^{\alpha}$ for an $\alpha>\frac{1}{3}$, then the energy is conserved.
\item[b)] For every $\alpha<\frac{1}{3}$ there exists a weak solution $u\in C^\alpha$ that dissipates energy.
\end{itemize}
\end{conjecture} 

It may seem strange that Onsager, a physicist, would worry about the validity of the chain rule for non-differentiable functions at a time where no counterexamples to the chain rule were known in the first place. In fact, he didn't; rather, his motivation stems from phenomenological turbulence theories (particularly the one of Kolmogorov) that make predictions on the energy spectrum of a fully turbulent fluid, which in turn lead to certain regularity properties of the velocity. This also means that non-conservative solutions of Euler are not always mathematical pathologies, but they are expected in turbulence theory and can be experimentally observed. We refer to~\cite{frisch, eyinksurvey, shvydkoysurvey} for more physically oriented overviews.

Onsager's conjecture has meanwhile largely been proved. Part b) was completed only recently in~\cite{isett16, buckmasteretal17}, based on convex integration techniques whose development started with~\cite{euler1}. Part a) is more classical -- it was initially solved in~\cite{GEY, CET} and then refined in~\cite{duchonrobert, CCFS08, FjWi}.

To formulate Part a) more precisely, we have the following result:
\begin{theorem}\label{eulerthm}
Let $(u,p)\in L^3\times L^{3/2}(\T^3\times(0,T))$ be a weak solution of~\eqref{inceuler} such that 
\begin{equation}\label{EulerBesov}
\lim_{\epsilon\to0} \int_0^T\int_{\T^3}\int_{B_\epsilon(0)}\frac{|u(x)-u(x-y)|^3}{\epsilon^{4}}dydxdt=0.
\end{equation} 
Then, the local energy inequality~\eqref{localen} is satisfied in the sense of distributions.
\end{theorem} 

To prove this, one proceeds exactly as in the proof of Theorem~\ref{transportthm}, where the velocity now plays the role of both $\rho$ and $u$, thus forcing $p=q=3$ and $\alpha=\beta=\frac13$.

As before, note that condition~\eqref{EulerBesov} is implied by $u\in L^3(0,T;C^{\alpha}(\T^3))$ for some $\alpha>\frac13$, or by $u\in L^3(0,T;B_{3,\infty}^{\alpha}(\T^3))$ for some $\alpha>\frac13$ (this is the space considered in~\cite{CET}), or by $u\in L^3(0,T;B_{3,c_0}^{1/3}(\T^3))$ (this is the space considered in~\cite{CCFS08}). Another important remark is that the result is actually true in any space dimension (so the exponent $1/3$ has nothing to do with 3D space!). In contrast, Part b) of Onsager's conjecture is still open in two dimensions.

\subsection{Bounded domains}
Let now $\Omega\in\R^3$ be a smooth bounded domain. Recall that the strategy for Theorems~\ref{transportthm} and~\ref{eulerthm} was to mollify the equation in space and thus to obtain
\begin{equation}\label{eulermolly}
\partial_t u_\epsilon+\Div(u_\epsilon\otimes u_\epsilon)+\nabla p_\epsilon=R_\epsilon
\end{equation}    
for the commutator $R_\epsilon=\Div(u_\epsilon\otimes u_\epsilon-(u\otimes u)_\epsilon)$. It seems like an issue that the mollification of a function cannot be defined in an obvious way on a bounded domain. However, choosing a compactly supported test function $\phi\in C^1_c(\Omega;\R)$, $u_\epsilon$ is well-defined on the support of $\phi$ as long as $\epsilon<\dist(\supp\phi,\partial\Omega)$. Therefore, testing~\eqref{eulermolly} against $\phi$, everything is well-defined. Suppose now~\eqref{EulerBesov} is satisfied on $\supp\phi$ for each $\phi\in C^1_c(\Omega)$; this means that~\eqref{EulerBesov} holds locally in $\Omega$, but not necessarily uniformly up to the boundary. Then, just as in the proof of Theorem~\ref{transportthm} we obtain $u_\epsilon\cdot R_\epsilon\to0$ in the sense of distributions, so that the local energy equality holds in the sense of distributions on $\Omega$.

The idea is that the local energy equality~\eqref{localen} is stronger than the global one~\eqref{globalen}, since one can deduce the latter from the former simply by integrating in space. On bounded domains, though, one should be a bit more careful. 

Given~\eqref{localen}, integration in space would amount to testing with $\phi\equiv1$. But we are only allowed to test with $\phi\in C^1_c(\Omega)$. So instead, let $\xi:\R_0^+\to\R_0^+$ be a smooth cutoff function that is zero near $s=0$ and one for $s\geq1$, and let $\xi_\delta(s):=\xi\left(\frac s\delta\right)$. Set $\phi_\delta(x):=\xi_\delta(\dist(x,\partial\Omega))$, so that $\phi_\delta\in C^1_c(\Omega)$, and test~\eqref{localen} against $\phi_\delta$:

\begin{equation}
\frac{d}{dt}\int_\Omega\phi_\delta\frac{|u|^2}{2}dx=\int_\Omega \nabla\phi_\delta\cdot u\left(\frac{|u|^2}{2}+p\right)dx. 
\end{equation}    
The left hand side converges to $\frac{d}{dt}\int_\Omega\frac{|u|^2}{2}dx$ as $\delta\to0$. If we want~\eqref{globalen} to hold, therefore, we must show that the right hand side converges to zero.

Since $\partial\Omega$ is smooth, there exists a neighbourhood $\Gamma_0$ of $\partial\Omega$ where the orthogonal projection of $x\in\Gamma_0$ to $\partial\Omega$ is uniquely defined. Let's call this projection $\sigma(x)\in\partial\Omega$, and denote by $n(\sigma(x))$ the outer unit normal to $\partial\Omega$ at $\sigma(x)$.

Then, $\nabla\phi_\delta(x)$ is parallel to $n(\sigma(x))$ with absolute value bounded by $\frac C\delta$, and on the other hand, $\nabla\phi_{\delta}$ is supported on $\Gamma_\delta:=\{\dist(x,\partial\Omega)<\delta\}$, where the volume of $\Gamma_\delta$ is comparable to $\delta$.

Hence, we have
\begin{equation}
\left|\int_\Omega \nabla\phi_\delta\cdot u\left(\frac{|u|^2}{2}+p\right)dx\right|\leq \frac C\delta\int_{\Gamma_\delta}\left|\frac{|u|^2}{2}+p\right||u(x)\cdot n(\sigma(x))|dx.
\end{equation}  
From this, it is easy to read off some sufficient conditions for global energy conservation: For instance, one could require $u$ and $p$ to be bounded in some neighbourhood of $\partial\Omega$, $x\mapsto u(x)\cdot n(\sigma(x))$ to be continuous in a neighbourhood of $\partial\Omega$, and $u\cdot n=0$ on $\partial\Omega$, which is the natural slip boundary condition for Euler anyway.

To summarise:
\begin{theorem}\label{eulerboundary}
Let $(u,p)$ be a weak solution of~\eqref{inceuler} such that $u$ satisfies~\eqref{EulerBesov} locally in $\Omega$. Suppose $(u,p)$ is bounded in some neighbourhood of $\partial\Omega$, $x\mapsto u(x)\cdot n(\sigma(x))$ is continuous in a neighbourhood of $\partial\Omega$, and $u\cdot n=0$ on $\partial\Omega$. Then, the global energy equality~\eqref{globalen} holds in the sense of distributions.
\end{theorem}
This result is taken from~\cite{BTW18}, where the theorem is stated under much weaker assumptions (in particular on the pressure). Similar results, with some differences on the technical level, were obtained independently in~\cite{DN18}. Earlier results on Onsager's conjecture in domains with boundaries were obtained in~\cite{RRS18, BT18}.

A useful feature of this result is that, near the boundary, it puts conditions only on the \emph{normal} component of the velocity, but not on the tangential one. This gives rise to an application to the viscosity limit. 

Recall the Navier-Stokes equations with viscosity $\nu>0$,
\begin{equation}\label{NSE}
\begin{aligned}
\partial_t u_\nu+(u_\nu\cdot\nabla)u_\nu+\nabla p_\nu&=\nu \Delta u_\nu\\
\Div u_\nu&=0,
\end{aligned}
\end{equation}   
which are known to admit (for given initial data in $L^2$) global weak (so-called Leray-Hopf) solutions in $\Omega$, subject to the no-slip boundary condition $u=0$ on $\partial\Omega$. It is commonly expected that, for small $\nu$, the Navier-Stokes flow will behave like an Euler flow except on a boundary layer of thickness of order $\sqrt\nu$, where it decays steeply to $u=0$ at $\partial\Omega$. Thus, one expects the normal velocity component to be very small in the boundary layer, but the tangential velocity component to have a gradient of magnitude $\nu^{-1/2}$. Since Theorem~\ref{eulerboundary} makes no assumption on the tangential component, the result is consistent with the formation of such a boundary layer. More precisely:

\begin{corollary}[\cite{BTW18}]
Let $(u_\nu)_{\nu>0}$ be a family of Leray-Hopf weak solutions of~\eqref{NSE} with viscosity $\nu$ and initial data $u^0\in L^2(\Omega)$, and suppose this family satisfies the conditions of Theorem~\ref{eulerboundary} uniformly in $\nu$. Then, there exists a subsequence $\nu_k\to0$ such that $u_\nu\to u$ strongly in $L^\infty(0,T;L^2(\Omega))$, where $u$ is a weak solution of the Euler equations that conserves energy.
\end{corollary} 

It is interesting to compare this result to~\cite{constvic, DN18a}, where purely interior assumptions are made on the solution, and anomalous energy dissipation is therefore not excluded.

\section{Statistical solutions}
As phenomenological theories of turbulence are of a statistical nature, it is natural to place the fundamental PDEs of fluid mechanics in a probabilistic framework. A classical way to do this is given by DiPerna's measure-valued solutions~\cite{DiPerna}, formulated for the incompressible Euler equations by DiPerna and Majda~\cite{dipernamajda}. To illustrate the idea, consider again the incompressible Euler equations~\eqref{inceuler}. Usually one wants to find a solution which is a vectorfield $u:\mathbb{T}^3\times[0,T]\to\R^3$, that is, at each (or at least almost every) point in space and time, one gives the velocity. In contrast, suppose the velocity at a point in space-time is not known exactly, but only as a probability distribution: This can be modelled by a map $\T^3\times[0,T]\to\mathcal P(\R^3)$ from space-time into the set of probability measures on the phase space $\R^3$. Thus, at each point $(x,t)$, there is a probability measure $\mu_{x,t}$ such that the probability that the velocity is in a (measurable) subset $U\subset\R^3$ is given by $\mu_{x,t}(V)$. 

Writing $\bar u(x,t):= \int_{\R^3}\xi d\mu_{x,t}(\xi)$ and $\overline{u\otimes u}(x,t):= \int_{\R^3}\xi\otimes\xi d\mu_{x,t}(\xi)$, we say that $\mu$ is a measure-valued solution of the Euler equations if 
 \begin{equation}
\begin{aligned}
\partial_t \bar u+\overline{u\otimes u}+\nabla p&=0\\
\Div \bar u&=0
\end{aligned}
\end{equation}   
   in the sense of distributions. (Observe that $(u\cdot\nabla)u=\Div(u\otimes u)$, which allows to write the equations in divergence form.) We deliberately ignore here the issue of possible concentrations requiring a generalised measure-valued setting, and refer to~\cite{dipernamajda}. 
	
Although such solutions are easily shown to exist for any initial datum, and have other useful properties (see e.g.~\cite{weakstrongsurvey}), they are felt by some to contain too little information. Indeed, it is not possible to represent with them two-point correlations, i.e.\ expressions of the form ``the probability that the velocity at point $(x,t)$ is in $U_1$ and the velocity at $(y,t)$ is in $U_2$''. On a related note, it is not obvious how to make sense of a Besov condition like~\eqref{BesovVMO} for a measure-valued solution.

To describe the statistics of a flow, Fjordholm et al.~\cite{FjLaMi} propose a new notion of statistical solutions, where solutions to conservation laws are given as \emph{correlation Young measures}. In the context of incompressible Euler, this was studied in~\cite{FjWi}. The idea is to describe the two-point statistics of a fluid in terms of a parametrised measure $\mu_{x,y,t}\in\mathcal{P}(\R^4\times\R^4)$, where $x,y\in \T^3$ are points from the space domain. The measure $\mu_{x,y}(du_1,dp_1,du_2,dp_2)$ can then be interpreted as the joint probability that the velocity and pressure at point $x$ are in $du_1\times dp_1$ and the velocity and pressure at point $y$ are in $du_2\times dp_2$. 

More precisely, such a correlation measure is a pair $(\mu^1,\mu^2)$, where $\mu^1_{x,t}$ is a probability measure on $\R^4$ and $\mu^2:{x,t}$ is a probability measure on $\R^4\times \R^4$ for almost every $(x,t)$, such that the following conditions are satisfied: 
\begin{enumerate}
\item \textit{Symmetry:} If $f\in C_0(\R^3\times\R^3)$ then 
\begin{equation*}
\int_{\R^4\times\R^4} f(\xi,\eta)d\mu^2_{x,y,t}(\xi, \eta) = \int_{\R^3\times\R^3} f(\eta,\xi)d\mu^2_{y,x,t}(\eta,\xi)
\end{equation*} for a.e.\ $x,y\in \R^3$; 
\item \textit{Consistency:} If $f\in C_0(\R^4\times\R^4)$ is of the form $f(\xi,\eta) = g(\xi)$ for some $g\in C_0(\R^4)$, then 
\begin{equation*}
\int_{\R^4\times\R^4} f(\xi,\eta)d\mu^2_{x,y,t}(\xi, \eta)=\int_{\R^4}g(\xi)d\mu^1_{x,t}(\xi)
\end{equation*}
for almost every $(x,y)\in \R^3\times \R^3$.
\end{enumerate}

One can now give a natural definition of $\mu$ to be a solution of the Euler equations: Indeed, if $(u,p)$ is a smooth solution and $u^i$ denotes the $i$-th velocity component, then 
\begin{equation}
\begin{split}
\partial_t\left(u^i(x)u^j(y)\right) + \sum_k \partial_{x^k}\left(u^i(x)u^k(x)u^j(y)\right) + \sum_k\partial_{y^k}\left (u^i(x)u^k(y)u^j(y)\right) \\
~ + \partial_{x^i}\left(p(x)u^j(y)\right) + \partial_{y^j} \left(u^i(x)p(y)\right) = 0,
\end{split}
\end{equation}
as can be seen by adding the equation evaluated at $x$ and multiplied by $u^j(y)$ to the equation evaluated at $y$ and multiplied by $u^i(x)$. Thus, replacing the occurrences of the velocity and pressure by the correlation measure, one is led to the equation 
\begin{equation}
\begin{split}
\partial_t\langle\mu^2_{x,y}u_1^iu_2^j\rangle + \sum_k \partial_{x^k}\langle\mu^2_{x,y}u_1^iu_1^ku_2^j\rangle + \sum_k\partial_{y^k}\langle\mu^2_{x,y}u_1^iu_2^ku_2^j\rangle \\
~ + \partial_{x^i}\langle\mu^2_{x,y}p_1u_2^j\rangle + \partial_{y^j} \langle\mu^2_{x,y}u_1^ip_2\rangle = 0,
\end{split}
\end{equation}
where we wrote $\langle\mu^2_{x,y}u_1^iu_2^j\rangle=\int_{\R^4\times\R^4}u_1^iu_2^jd\mu_{x,y}(u_1,p_1,u_2,p_2)$ etc.\ (hence the dummy variables $u_1, u_2, p_1, p_2$ stand for $u(x), u(y), p(x), p(y)$, respectively).

The natural extension of the Besov-type assumption~\eqref{EulerBesov} is then 
 \begin{equation}
\lim_{\epsilon\to0} \int_0^T\int_{\T^3}\int_{B_\epsilon(0)}\frac{\langle \mu^2_{x,x-y}; |u_1-u_2|^3\rangle}{\epsilon^{4}}dydxdt=0,
\end{equation} 
which is the measure-valued analogue of~\eqref{EulerBesov}.

In~\cite{FjWi} we show that the energy is (locally) conserved under this assumption.

\section{General conservation laws}
A crucial observation made in~\cite{FGSW} is that the commutator estimates remain valid even when the nonlinearities involved are not quadratic. For instance, for the isentropic compressible Euler system 
\begin{equation}\label{compE}
\begin{aligned}
\partial_t(\rho u)+\Div(\rho u\otimes u)+\nabla p(\rho)&=0,\\
\partial_t \rho+\Div(\rho u)&=0,
\end{aligned}
\end{equation}
one formally has the local conservation of energy:
\begin{equation*}
\partial_t\left(\frac{\rho |u|^2}{2}+P(\rho)\right)+\operatorname{div}\left[\left(\frac{\rho |u|^2}{2}+P(\rho)+p(\rho)\right)v\right]= 0,
\end{equation*}
where $P$ is the pressure potential as defined in~\eqref{presspot}. To prove this rigorously for weak solutions with a commutator argument, one needs to estimate, e.g., 
\begin{equation}
p(\rho)_\epsilon-p(\rho_\epsilon).
\end{equation}
If $\rho\mapsto p(\rho)$ is twice differentiable in the closure of the range of $\rho$, Taylor expansion yields
\begin{equation}
p(\rho_\epsilon)\sim p(\rho) + p'(\rho)(\rho_\epsilon-\rho) + \frac12 p''(\rho)(\rho_\epsilon-\rho)^2 
\end{equation}
as well as 
\begin{equation}
p(\rho(y))\sim p(\rho(x)) + p'(\rho(x))(\rho(y)-\rho(x)) + \frac12 p''(\rho(x))(\rho(y)-\rho(x))^2. 
\end{equation}
Multiplication of the latter by $\chi_\epsilon(x-y)$, integration w.r.t.\ $y$, and subtraction of both equations yields
\begin{equation*}
|p(\rho_\epsilon)-p(\rho)_\epsilon|\lesssim  \|p\|_{C^2}(\rho_\epsilon-\rho)^2,
\end{equation*}
so we have reduced the problem again to a quadratic nonlinearity. In this way, one arrives at the following result:

\begin{theorem}\label{compressibleonsager}
Let $\rho$, $v$ be a solution of~\eqref{compE} in the sense of distributions. Assume $\rho$ and $\rho v$ both satisfy~\eqref{BesovVMO} and $v$ satisfies~\eqref{BesovVMO2}, where also a time shift is considered\footnote{This means that~\eqref{BesovVMO} turns into \begin{equation*}
\lim_{\epsilon\to0} \int_0^T\int_{\T^d}\int_{B_\epsilon(0)}\frac{|\rho(x,t)-\rho(x-y,t-\tau)|^p}{\epsilon^{d+1+\alpha p}}dyd\tau dxdt=0,
\end{equation*} where the ball with radius $\epsilon$ is considered in space-time. Analogously, one needs to alter~\eqref{BesovVMO2}.}\label{footnote}. Suppose further
\begin{equation*}
0 \leq \underline{\rho} \leq \rho \leq \overline{\rho} \ \mbox{a.a. in} (0,T)\times\mathbb{T}^d,
\end{equation*}
for some constants $\underline{\vr}$, $\overline{\vr}$, and

\begin{equation}
2\alpha+\beta>1,\quad \alpha+2\beta>1, \quad p=q=3.
\end{equation}
Assume moreover that $p \in C^2[\underline{\vr}, \overline{\vr}]$, and, in addition
\begin{equation}
p'(0) = 0 \ \mbox{as soon as}\ \underline{\vr} = 0.
\end{equation}
Then the energy is locally conserved, i.e.
\begin{equation*}
\partial_t\left(\frac{1}{2}\rho|u|^2+P(\rho)\right)+\diverg\left[\left(\frac{1}{2}\rho|u|^2+p(\rho)+P(\rho)\right)u\right]=0
\end{equation*}
in the sense of distributions on $(0,T)\times\mathbb{T}^d$.
\end{theorem}

Some remarks are in order. The result bears resemblance to Theorem~\ref{transportthm}, the most important difference arguably being the symmetry between the regularity and integrability indices $\alpha, \beta, p, q$; this comes from the fact that, unlike for the transport equation, there are now several commutators to control, in some of which the terms $\rho, \rho u$ appear twice and $u$ appears once, and vice versa. Secondly, we need to control the Besov-type regularity now also in time, since there is the nonlinear term $\rho u$ under the time derivative. (A possible way to avoid this is to write~\eqref{compE} in conservative variables, thus replacing $\rho u$ by $m$, and to obtain the energy equality upon testing with $\frac{m_\epsilon}{\rho_\epsilon}$ rather than $u_\epsilon$. This idea, carried out in~\cite{LSh} for the inhomogeneous incompressible Euler equations, succeeds in avoiding any assumption on time regularity, but leads to trouble with vacuum states.) 

The assumption $p\in C^2$ in the closure of the range of $\rho$ will typically require \emph{absence of vacuum}: A common choice for the pressure is the polytropic law $p(\rho)=\rho^\gamma$, $\gamma>1$, and often $\gamma\leq 5/3$ (the exponent $5/3$ corresponds to a monoatomic gas). As such a function $p$ is twice differentiable only away from $\rho=0$, the assumption to justify the Taylor expansion boils down to absence of vacuum states. Note that this property (namely, $\rho\geq c>0$) is not necessarily propagated in time: Suppose the initial density satisfies $\rho_0\geq c>0$, then the maximum principle for transport equations~\cite{dipernalions} implies that $\rho$ will remain bounded away from zero if the divergence of the velocity is bounded. This, however, is not necessarily satisfied for weak solutions of~\eqref{compE} (not even for the compressible Navier-Stokes system). We will get back to the vacuum problem in the next section.

More generally, the Taylor expansion strategy applies to essentially any system of conservation laws that possesses an entropy~\cite{GwMiSw, BTWP218, BGSTW}. To illustrate the point, let
\begin{equation}
\partial_t u+\partial_x f(u)=0
\end{equation}
be a scalar conservation law in one dimension, so that $u:\R\times[0,T]\to\R$ is the unknown and $f$ is a given smooth flux function. Let $\eta:\R\to\R$ be any convex function (an \emph{entropy}) and $q:\R\to\R$ a corresponding \emph{entropy flux}, meaning that $q'=f'\eta'$. Then it is easy to see that, again by the chain rule, the \emph{entropy equality}
\begin{equation}
\partial_t\eta(u)+\partial_xq(u)=0
\end{equation} 
holds, at least if $u$ is smooth. (The entropy equality is really the same thing as the renormalised equation~\eqref{renorm} for transport equations, or the local energy equality~\eqref{localen} for the Euler equations.) 

Using the arguments outlined above, one can then show that this is the case if~\eqref{EulerBesov} holds. Thus, the exponent $1/3$ appears universally, simply because the leading order in the Taylor expansion of a smooth function which does not commute with a mollification is the second order. Remarkably, shocks provide an easy example that the 1/3-condition~\eqref{EulerBesov} is optimal.

\section{Degenerate cases}
The reduction of an arbitrary nonlinearity to a quadratic one relies crucially on the boundedness of second derivatives for the nonlinearities and entropies involved. There are at least two interesting cases when this condition fails: First, when a transport equation is to be renormalised with $\eta(\rho)=|\rho|^p$ with $p<2$, and secondly when one considers the compressible Euler system with possible vacuum, for pressure laws $p(\rho)=\rho^\gamma$ with $1<\gamma<2$ (these are the physically interesting ones).

These problems have been studied in~\cite{AkWi} and~\cite{AkDeSkWi}, respectively. We give here a brief outline of the latter. Recall the discussion of the compressible Euler system in the previous section. One idea is to approximate $p$ locally uniformly by a function $p^\delta$ with bounded second derivatives (up to $\rho=0$ of course). The main error term introduced by this additional layer of approximation takes the form
\begin{equation*}
\int_{\mathbb{T}^3}u_\eps \cdot\nabla[p^\delta(\rho)_\eps-p(\rho)_\eps]dx=-\int_{\mathbb{T}^3}\operatorname{div}u_\eps [p^\delta(\rho)_\eps-p(\rho)_\eps]dx,
\end{equation*}
and thus converges to zero, uniformly in $\epsilon$, provided $\operatorname{div} u$ is a bounded measure. While for the Euler equations this condition (which is quite popular in the theory of hyperbolic conservation laws, see~\cite{ChenFrid}) cannot be guaranteed a priori, for the compressible Navier-Stokes equations~\eqref{compNSE} it follows directly from the energy estimate. We therefore obtain that the compressible Navier-Stokes equations conserve energy, even with possible vacuum, under the assumptions of Theorem~\ref{compressibleonsager} except the $C^2$ condition on the pressure law. This yields a nice complement to the reuslts in~\cite{ChengYu}.

If one is not prepared to make such an assumption on $\operatorname{div}u$, one needs to consider two commutator terms involving the pressure, which are therefore sensitive to the condition that $p''$ be bounded. These two commutators, which appear in the course of the computation, are
\begin{equation*}
R^1_\eps:=\int_{\mathbb{T}^3}\operatorname{div} u_\eps(p(\rho_\eps)-p(\rho)_\eps)dx,
\end{equation*}
and
\begin{equation*}
R^2_\eps:=\int_{\mathbb{T}^3}\operatorname{div} (\rho_\eps u_\eps-(\rho v)_\eps) P'(\rho_\eps)dx.
\end{equation*}

It is not difficult to show $R^1_\eps\to0$ as long as $\rho$, $u$ satisfy~\eqref{BesovVMO}, \eqref{BesovVMO2} (in the space-time sense of footnote~\ref{footnote}) with $p=q=3$ and
\begin{equation*}
\gamma\alpha+\beta>1,
\end{equation*}
which is a stronger version of the previous condition $2\alpha+\beta>1$. (To show convergence of $R^1_\eps$ under this condition, one uses ``Taylor expansion to order $\gamma$'', see~\cite[Lemma 4.2]{AkDeSkWi}.)

The other commutator $R^2_\eps$ is more delicate. We compute
\begin{equation*}
\begin{aligned}
R^2_\eps&=\int_{\mathbb{T}^3}\operatorname{div} (\rho_\eps u_\eps-(\rho u)_\eps) P'(\rho_\eps)dx\\
&=- \int_{\{\rho_\eps>0\}} (\rho_\eps u_\eps-(\rho u)_\eps) \cdot P''(\rho_\eps)\nabla\rho_\eps dx\\
&\sim \int_{\{\rho_\eps>0\}} (\rho_\eps-\rho) (u_\eps- u) \cdot \rho_\eps^{\gamma-2}\nabla\rho_\eps dx,
\end{aligned}
\end{equation*}
then split the domain of integration into $B^\eps:=\{0<\rho_\eps<\eps^\alpha\}$ and $C^\eps :=\{\rho_\eps\geq\eps^\alpha\}$, and only consider integration over $B^\eps$ here (the $C^\eps$ part is easier):

\begin{equation}\label{333}
\begin{aligned}
&\left|\int_{B^\eps} (\rho_\eps-\rho) (u_\eps- u) \cdot \rho_\eps^{\gamma-2}\nabla\rho_\eps dx\right|\\
&\leq \int_{B^\eps} \left|\frac{\rho_\eps-\rho}{\rho_\eps}\right| |u_\eps- u| \left|\rho_\eps^{\gamma-1}\right||\nabla\rho_\eps| dx\\
&\leq \eps^{\alpha(\gamma-1)}\int_{\mathbb{T}^3} \left|\frac{\rho_\eps-\rho}{\rho_\eps}\right| |u_\eps- u| |\nabla\rho_\eps| dx\\
&\leq \eps^{\alpha(\gamma-1)}\|u_\eps-u\|_{L^3}\|\nabla\rho_\eps\|_{L^3}\left\|\frac{\rho_\eps-\rho}{\rho_\eps}\right\|_{L^3}\\
&\lesssim  \eps^{\gamma\alpha+\beta-1}\left\|\frac{\rho_\eps-\rho}{\rho_\eps}\right\|_{L^3}.
\end{aligned}
\end{equation}

Now all is well as long as $\left\|\frac{\rho_\eps-\rho}{\rho_\eps}\right\|_{L^3}$ is bounded uniformly in $\eps$; however, in general, this may be false: If $p>1$, then there are smooth non-negative functions $\rho$ such that $\left\|\frac{\rho_\eps-\rho}{\rho_\eps}\right\|_{L^p}\to\infty$ as $\epsilon\searrow 0$ \cite[Subsection 4A]{AkDeSkWi}. Only for $p=1$ are we able to control this term~\cite[Lemma 4.3]{AkDeSkWi}.

This suggests we should use a H\"older $\infty-\infty-1$ estimate instead of a $3-3-3$ estimate in~\eqref{333}; for this, in turn, we need to require $\rho$ and $v$ to be H\"older continuous. We thus arrive at the following result:

\begin{theorem}
Replace $2\alpha+\beta$ with $\gamma\alpha+\beta$ as well as~\eqref{BesovVMO} and~\eqref{BesovVMO2} with $\rho\in C^\alpha(\mathbb{T}^3\times(0,T))$ and $u\in C^\beta(\mathbb{T}^3\times(0,T))$ in the assumptions of Theorem~\ref{compressibleonsager}, and allow for $1<\gamma<2$, $\rho\geq0$. Then the energy is conserved.
\end{theorem}

\end{document}